\documentclass[11pt]{article}
\usepackage{amsmath, amssymb, amsthm, fullpage}
\usepackage{tikz}
\date{}
\title{\vspace{-1cm} $K_{s,t}$-saturated bipartite graphs
}
\author{
Wenying Gan \thanks{Department of Mathematics, ETH, 8092 Zurich, Switzerland. Email: wenying.gan@math.ethz.ch.}
\and
D\'aniel Kor\'andi \thanks{Department of Mathematics, ETH, 8092 Zurich, Switzerland. Email: daniel.korandi@math.ethz.ch.}
\and
Benny Sudakov \thanks{Department of Mathematics, ETH, 8092 Zurich, Switzerland and Department of Mathematics, UCLA, Los Angeles, CA 90095.
Email: benjamin.sudakov@math.ethz.ch. Research supported in part by SNSF grant 200021-149111 and by a USA-Israel BSF grant.}
}

\theoremstyle{plain}
\newtheorem{THM}{Theorem}[section]
\newtheorem*{THM*}{Theorem}
\newtheorem{PROP}[THM]{Proposition}
\newtheorem{LEMMA}[THM]{Lemma}

\newtheorem{COR}[THM]{Corollary}
\newtheorem{CLAIM}[THM]{Claim}
\newtheorem{CONJ}[THM]{Conjecture}

\theoremstyle{definition}

\newtheorem*{example}{Example}

\newcommand{\floor}[1]{\left\lfloor #1 \right\rfloor}

\newcommand{\subs}{\subseteq}

\begin{document}
\maketitle

\begin{abstract}

An $n$-by-$n$ bipartite graph is $H$-saturated if the addition of any missing edge between its two parts creates a new copy 
of $H$. In 1964, Erd\H{o}s, Hajnal and Moon made a conjecture on the minimum number of edges in a $K_{s,s}$-saturated 
bipartite graph. This conjecture was proved independently by Wessel and Bollob\'as in a more general, but ordered, setting: 
they showed that the minimum number of edges in a $K_{(s,t)}$-saturated bipartite graph is $n^2-(n-s+1)(n-t+1)$, where 
$K_{(s,t)}$ is the ``ordered'' complete bipartite graph with $s$ vertices in the first color class and $t$ vertices in the 
second. However, the very natural question of determining the minimum number of edges in the unordered $K_{s,t}$-saturated 
case remained unsolved. This problem was considered recently by Moshkovitz and Shapira who also
conjectured what its answer should be. In this short paper we 
give an asymptotically tight bound on the minimum number of edges in a $K_{s,t}$-saturated bipartite graph, which is only 
smaller by an additive constant than the conjecture of Moshkovitz and Shapira. We also prove their conjecture for 
$K_{2,3}$-saturation, which was the first open case.

\end{abstract}

\section{Introduction} \label{sec:introduction}

For two graphs $G$ and $H$, $G$ is said to be $H$-saturated if it contains no copy of $H$ as a subgraph, but the addition of any edge missing from $G$ creates a copy of $H$ in $G$. The saturation number $sat(n,H)$ is defined as the minimum number of edges in an $H$-saturated graph on $n$ vertices. Notice that the problem of finding the saturation number for $H$ is, in some sense, the dual of the classical Tur\'an problem.

Probably the most natural setup of this problem is when we choose $H$ to be a fixed complete graph $K_s$. This was first studied by Zykov \cite{Z49} in the 1940's, and later by Erd\H{o}s, Hajnal and Moon \cite{EHM64} in 1964. They proved that $sat(n,K_s)= (s-2)n - \binom{s-1}{2}$. Here the upper bound comes from the $K_s$-saturated graph that has $s-2$ vertices connected to all other vertices. Later, the closely related notion of weak saturation was introduced by Bollob\'as \cite{B68}. A graph $G$ is weakly $H$-saturated if it is possible to add back the missing edges of $G$ one by one in some order, so that each addition creates a new copy of $H$. Trivially, if $G$ is $H$-saturated, then any order satisfies this property, hence $G$ is also weakly $H$-saturated. Let $w$-$sat(n,K_s)$ be the minimum number of edges in an $n$-vertex graph that is weakly $K_s$-saturated. We then have $w$-$sat(n,K_s) \leq sat(n,K_s)$. Somewhat surprisingly, one can prove using algebraic techniques (see e.g. \cite{L77}) that these two functions are actually equal. On the other hand, the extremal graphs for these problems are not the same and already for $s=3$ there are weakly $K_3$-saturated graphs (i.e. trees) which are not $K_3$-saturated.

The paper by Erd\H{o}s, Hajnal and Moon also introduced the bipartite saturation problem, where we are looking for the minimum number of edges $sat(K_{n,n},H)$ in an $H$-free $n$-by-$n$ bipartite graph, such that adding any missing edge between the two color classes creates a new copy of $H$. (Of course, this definition is only meaningful if $H$ is also bipartite.) They conjectured that $sat(K_{n,n},K_{s,s}) = n^2 - (n-s+1)^2$. Once again, this is seen to be tight by selecting $s-1$ vertices on each side of the bipartite graph and connecting them to every vertex on the opposite side.  In the bipartite setting, one can impose an additional restriction on the problem by ordering the two vertex classes of $H$ and requesting that each missing edge create an $H$ respecting the order: the first class of $H$ lies in the first class of $G$. For example, let $K_{(s,t)}$ be the complete ``ordered'' $s$-by-$t$ bipartite graph with $s$ vertices in the first class and $t$ vertices in the second, then a bipartite graph $G$ is $K_{(s,t)}$-saturated if each missing edge creates a $K_{s,t}$ with the $s$-vertex class lying in the first class of $G$. Indeed, the conjecture of Erd\H{o}s, Hajnal and Moon was independently confirmed by Wessel \cite{W66} and Bollob\'as \cite{B67} a few years later as the special case of the following result: $sat(K_{(n,n)},K_{(s,t)}) = n^2-(n-s+1)(n-t+1)$. This was further generalized in the 80s by Alon \cite{A85} to complete $k$-uniform hypergraphs in a $k$-partite setting using algebraic tools.
Alon showed that the saturation and weak saturation bounds are the same in this case as well. For a more detailed discussion of $H$-saturation in general, we refer the reader to the survey \cite{FFS} by Faudree, Faudree and Schmitt.

In this paper we study the unordered case of bipartite saturation. Although this is arguably the most natural setting for the bipartite problem, it did not receive any attention until very recently in \cite{BBMR, MS12}. Moshkovitz and Shapira \cite{MS12} studied the unordered weak saturation number of $K_{s,t}$, $s \le t$, and showed that $w$-$sat(K_{n,n},K_{s,t}) = (2s-2+o(1))n$. Note that, surprisingly, it is much smaller than the corresponding ordered saturation number and only depends on the size of the smaller part. One might think that a similar gap exists for saturation numbers as well.
Moshkovitz and Shapira \cite{MS12} conjectured that this is not the case, and that ordered and unordered bipartite saturation numbers differ only by an additive constant. More precisely, they made the following conjecture and constructed an example showing that, if true, this bound is tight.

\begin{CONJ}\label{conj:main}
Let $1\le s\le t$ be integers. Then there is an $n_0$ such that if $n\ge n_0$ and $G$ is a $K_{s,t}$-saturated $n$-by-$n$ bipartite graph, then $G$ contains at least $(s+t-2)n- \floor{ \left( \frac{s+t-2}{2} \right)^2 }$ edges.
\end{CONJ}

In this short paper we prove the following result which confirms the above conjecture up to a small additive constant.
\begin{THM}\label{THMmain}
Let $1\leq  s\leq t$ be fixed and $n\ge t$. Then
\[ sat(K_{n,n},K_{s,t}) \ge (s+t-2)n - (s+t-2)^2.\]
\end{THM}

\noindent
The proof is presented in Section \ref{sec:main}. In Section \ref{sec:example}, we show that if the conjecture is true, it has many extremal examples. Finally,
in Section \ref{sec:case}, we prove Conjecture \ref{conj:main} in the first open case of $K_{2,3}$-saturation.

\section{Lower bounds on the saturation number}  \label{sec:main}

Let $G$ be a bipartite graph with vertex class $U$ and $U'$ of size $n$. Assume $1\le s\le t\le n$ and suppose $G$ is $K_{s,t}$-saturated, i.e. each missing edge between $U$ and $U'$ creates a new $K_{(s,t)}$ or a new $K_{(t,s)}$ when added to $G$. Here $K_{(a,b)}$ refers to a complete bipartite graph with $a$ vertices in $U$ and $b$ vertices in $U'$.

Let us start with the following, easy special case of Theorem \ref{THMmain}.

\begin{PROP} \label{prop:smalldeg}
Suppose a $K_{s,t}$-saturated graph has minimum degree $\delta<t-1$. Then it contains at least $n(t+s-2) - (s+t-2)^2$ edges.
\end{PROP}
\begin{proof}
The $K_{s,t}$-saturated property ensures that each vertex has at least $s-1$ neighbors, so we actually have $s-1 \le \delta < t-1$. Let $u_0$ be a vertex of degree $\delta$; we may assume that $u_0\in U$. Then adding any missing edge $u_0u'$ to $G$ (where $u'\in U'-N(u_0)$) should create a new $K_{(t,s)}$ because it cannot create a $K_{(s,t)}$. For such a $u'$, let $S_{u'} \subs U$ be the $t-1$ vertices other than $u_0$ in the $t$-class of this $K_{(t,s)}$, and define $V\subs U$ to be the union of these $S_{u'}$. Then all vertices in $V$ have at least $s-1$ neighbors in $N(u_0)$ and all vertices in $U' - N(u_0)$ have at least $t-1$ neighbors in $V$. Now we can count the number of edges in $G$ as follows:
\begin{align*}
e(U,U') &= e(V,N(u_0)) + e(V,U' - N(u_0)) + e(U - V, U')\\
&\ge (s-1)|V| + (t-1)(n-|N(u_0)|) + \delta(n-|V|) \\
&\ge (s-1)|V| + (t-1)(n-t+2) + (s-1)(n-|V|) \\
&\ge n(t+s-2) - (t-1)(t-2)\\
&\ge n(t+s-2) - (s+t-2)^2.
\end{align*}
\end{proof}

The case when $\delta\ge t-1$ is considerably more complicated. We introduce the following structure to count the edges of $G$ (see Figure 1). The \emph{core} of this structure is a set $\tilde{A_0}=A_0\cup A_0'$ with $A_0\subs U$ and $A_0'\subs U'$ satisfying the following technical property:
\begin{itemize}
  \item there are vertices $u_0\in A_0$ and $u_0'\in A_0'$ such that their neighborhoods are also contained in the core.
\end{itemize}

Next, we build the shell around the core: starting with $\tilde{A}=\tilde{A_0}$, we iteratively add any vertex $v$ to $\tilde{A}$ that has at least $t-1$ neighbors in it. In other words, $\tilde{A}=A\cup A'$ is the smallest set containing $\tilde{A_0}$ such that any vertex $v\in G-\tilde{A}$ has fewer than $t-1$ neighbors in $\tilde{A}$. Here $A_0\subs A\subs U$ and $A_0'\subs A'\subs U'$. We use the variables $x_0=|A_0|$, $x_0'=|A_0'|$, $x=|A|$ and $x'=|A'|$ to denote the sizes of the corresponding sets. Obviously $x_0\le x$ and $x_0'\le x'$.

\medskip

The following, rather scary, lemma is the key to our lower bounds on the saturation numbers. It shows that we can find about $n(s+t-2)$ edges in a $K_{s,t}$-saturated graph, provided we have a small enough core.

\begin{figure}[t!]
\begin{center}
\begin{tikzpicture}

\draw (2,3.5) ellipse (1.5cm and 0.75cm);
\node at (2,2) {$\boldsymbol{A'}$};

\draw (2,6) ellipse (1.5cm and 0.75cm);
\node at (2,7.5) {$\boldsymbol{A}$};

\draw [fill] (0.64, 3.5) circle (0.04);
\node[left] at (0.6,3.5) {$u_0'$};

\draw [fill] (0.64, 6) circle (0.04);

\node[left] at (0.6,6) {$u_0$};

\draw (1.5, 3.5) ellipse (1cm and 0.5cm);
\node at (1.5, 3.5) {$\boldsymbol{A_0'}$};
\draw (1.5, 6) ellipse (1cm and 0.5cm);
\node at (1.5, 6) {$\boldsymbol{A_0}$};

\draw (7, 3.5) ellipse (3cm and 0.75cm);
\draw (7, 2.75)--(7, 4.25);

\node at (5.9,3.5) {$\boldsymbol{B_1'}$};

\node at (8.1,3.5) {$\boldsymbol{B_2'}$};

\node at (7, 2) {$\boldsymbol{B'}$};

\draw (7, 6) ellipse (3cm and 0.75cm);
\draw (7, 5.25)--(7, 6.75);

\node at (5.9,6.2) {$\boldsymbol{B_1}$};
\node at (5.9,5.7) {\footnotesize $d_{B'}\ge t-s$};

\node at (8.1,6.2) {$\boldsymbol{B_2}$};
\node at (8.1,5.7) {\footnotesize $d_{B'} < t-s$};

\node at (7,7.5) {$\boldsymbol{B}$};
\node at (7,7) {\footnotesize $s-1\leq d_{A'}<t-1$};

\draw (13.5,3.5) ellipse (3cm and 0.75cm);
\draw (13.5,2.75)--(13.5,4.25);
\node at (12.2,3.5) {$\boldsymbol{C_1'}$};
\node at (14.8,3.5) {$\boldsymbol{C_2'}$};
\node at (13.5,2) {$\boldsymbol{C'}$};

\draw (13.5, 6) ellipse (3cm and 0.75cm);
\draw (13.5, 5.25)--(13.5, 6.75);
\node at (12.2,6.2) {$\boldsymbol{C_1}$};
\node at (12.2,5.7) {\footnotesize $d_{V-B_2'}\ge s-1$};
\node at (14.8,6.2) {$\boldsymbol{C_2}$};
\node at (14.8,5.7) {\footnotesize $d_{V-B_2'}<s-1$};
\node at (13.5,7.5) {$\boldsymbol{C}$};
\node at (13.5,7) {\footnotesize $d_{A'}<s-1$};

\end{tikzpicture}
\caption{the structure for counting the edges}
\end{center}
\end{figure}
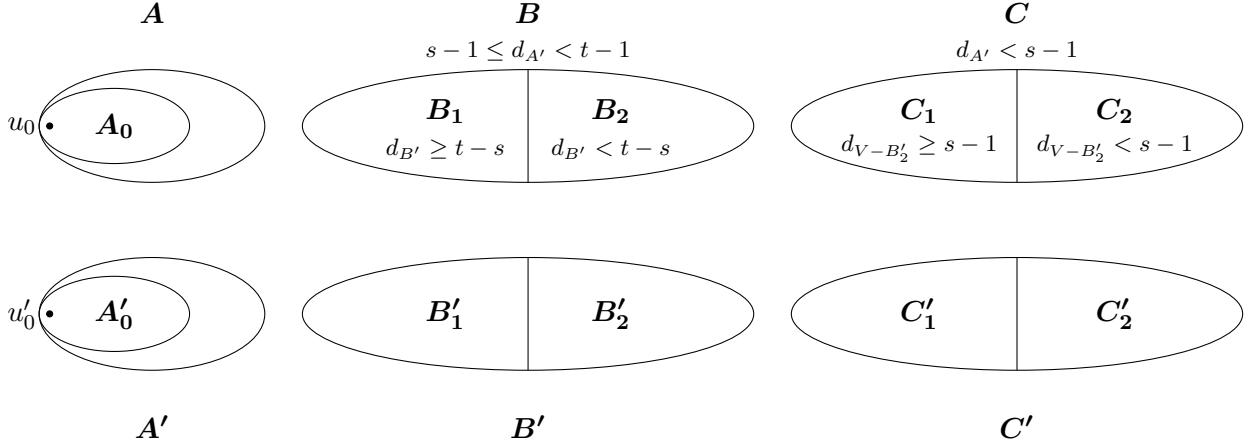

\begin{LEMMA} \label{lem:main}
Assuming $\delta\ge t-1$, suppose the core spans $e=e(A_0,A_0')$ edges. Then $G$ has at least
\[ n(s+t-2) - (x_0+x_0')(t-1) - \floor{\frac{(s-1)^2}{4}} + e + \min\{(t-s)x,(t-s)x'\} \]
edges.
\end{LEMMA}

\begin{proof} By the construction of $\tilde{A}$, we know that it spans at least $e+(t-1)(x+x'-x_0-x_0')$ edges. Indeed, each vertex we added to the shell brings at least $t-1$ new edges. The idea is to count $t-1$ edges from the remaining vertices on one side of the graph, say $U'-A'$, and then to find $s-1$ new (yet uncounted) edges from the other side, $U-A$. Of course, if a vertex in $U-A$ has at least $s-1$ neighbors in $A'$, then these edges are guaranteed to be new.

So let us continue with our definition of the structure. We know that any vertex in $U-A$ has fewer than $t-1$ neighbors in $A'$. We break this set into two parts by defining $B$ to be the set of vertices in $U-A$ having at least $s-1$ neighbors in $A'$, and $C$ to be those having fewer than $s-1$ neighbors in $A'$. Similarly we break $U'-A'$ into two sets $B'$ and $C'$ based on the size of the neighborhood in $A$. We need to break $B$ further into two parts $B_1$ and $B_2$, by defining $B_1$ to be the set of vertices having at least $t-s$ neighbors in $B'$. Similarly, let $B_1'$ be the set of vertices in $B'$ having at least $t-s$ neighbors in $B$ (see Figure 1).

Note that any vertex in $B_1'$ already has $t-1$ neighbors in $A\cup B$ (at least $s-1$ in $A$ and at least $t-s$ in $B$), but this is not necessarily true for $B_2'$. This, together with our strategy to find $s-1$ new edges from the vertices in $C$ motivates our last partitioning: We now break $C$ into two parts $C_1$ and $C_2$, where $C_1$ is the set of those vertices in $C$ which have at least $s-1$ neighbors outside $B_2'$, and $C_2=C-C_1$. We similarly define $C_1'=\{v\in C': |N(v)-B_2|\ge s-1\}$, where $N(v)$ is the neighborhood of $v$, and $C_2'=C'-C_1'$.

An observation here, which will prove to be crucial when counting the edges, is that $C_2$ and $C_2'$ span a complete bipartite graph. Indeed, suppose there is a missing edge $vv'$ in $G$, with $v\in C_2$ and $v'\in C_2'$. Adding this edge creates a $K_{(s,t)}$ or a $K_{(t,s)}$, suppose it is a $K_{(s,t)}$. Then $v'$ is connected to all the $s-1$ vertices other than $v$ in the $s$-vertex class of this $K_{(s,t)}$. But $v'$ is in $C_2'$, so it has at most $s-2$ neighbors outside $B_2$, consequently there is a vertex $w\in B_2$ in the $s$-class. Similarly, using that $v$ is in $C_2$, we find at least $t-s$ vertices of the $t$-class in $B_2'$. But then $w\in B_2$ has at least $t-s$ neighbors in $B_2'\subs B'$, which contradicts the definition of $B_2$. The same argument leads to a contradiction if the edge creates a $K_{(t,s)}$, hence we can conclude that there is no missing edge between $C_2$ and $C_2'$.

On another note, observe that adding the edge $u_0v'$, where $u_0$ is the vertex in $A_0$ defined in the property of the core and $v'$ is any vertex in $C'$, cannot create a $K_{(s,t)}$. Indeed, if it created a $K_{(s,t)}$, then all the vertices of the $t$-class except $v'$ are neighbors of $u_0$, so they are sitting in the core, $A_0'$. This means that each vertex in the $s$-class is connected to at least $t-1$ vertices in the core, hence the whole $s$-class is in $A$. But then $v'$ has at least $s-1$ neighbors in $A$, contradicting $v'\in C'$. So we see that adding $u_0v'$ creates a $K_{(t,s)}$. Then, all the vertices of the $s$-class of this copy of
$K_{(t,s)}$ except $v'$ are in $A_0'$, therefore the vertices of the $t$-class have at least $s-1$ neighbors in $A'$. Hence all of them are in $A\cup B$, implying that every $v'\in C'$ has at least $t-1$ neighbors in $A\cup B$. The same argument shows that each $v\in C$ has at least $t-1$ neighbors in $A'\cup B'$.

Lemma~\ref{lem:main} will now follow from the following claim, possibly applied to the graph with the two vertex classes switched.

\begin{CLAIM} \label{lem:sub}
Assuming $\delta\ge t-1$, suppose $|C_2|\le |C_2'|$. Then
\[ e(U,U') \ge n(s+t-2) - (x_0+x_0')(t-1) - \floor{\frac{(s-1)^2}{4}} + e + (t-s)x. \]
\end{CLAIM}
\begin{proof} Let $y=|C_2|$ and $y'=|C_2'|$, and let us count the edges in $G$. We noted above that each vertex in $B_1'$ has at least $t-1$ neighbors in $A\cup B$, so $e(A\cup B,B_1')\ge (t-1)|B_1'|$. By assumption, each vertex in $B_2'$ has degree at least $t-1$, hence $e(A\cup B\cup C, B_2')\ge (t-1)|B_2'|$. We have also shown that each vertex in $C'$ has at least $t-1$ neighbors in $A\cup B$, so $e(A\cup B,C')\ge (t-1)|C'|$. This so far means that
\begin{equation}    \label{eq:bottom}
e(A\cup B,B_1') + e(A\cup B\cup C, B_2') + e(A\cup B, C') \ge (t-1)(n-x').
\end{equation}
Now look at what we have left from the other side: By definition, any vertex in $B$ has at least $s-1$ neighbors in $A'$, so $e(B,A')\ge (s-1)|B|$. We also defined $C_1$ so that its vertices have at least $s-1$ neighbors outside $B_2'$, this gives $e(C_1,A'\cup B_1'\cup C')\ge (s-1)|C_1|$. As we noted above, the vertices of $C_2$ are all connected to the vertices of $C_2'$, so $e(C_2,C_2')=yy'$. Using the fact that $y(s-1-y)\le \floor{\frac{(s-1)^2}{4}}$ ($y$ is an integer), we get that
\begin{align}    \label{eq:top}
e(B,A') + e(C_1,A'\cup B_1'\cup C') + e(C_2,C_2') &\ge (s-1)(n-x-y) + yy' \\
    &\ge (s-1)(n-x) - (s-1)y +y^2     \nonumber \\
    &\ge (s-1)(n-x) - \floor{\frac{(s-1)^2}{4}}.    \nonumber
\end{align}
We have also seen that $e(A,A')$ is at least $e+(t-1)(x+x'-x_0-x_0')$.

It is easy to check that we never counted an edge more than once above, hence
\begin{align}
e(U,U') &\ge (t-1)(n-x') + (s-1)(n-x) + (t-1)(x+x'-x_0-x_0') + e - \floor{\frac{(s-1)^2}{4}} \\
        &= n(t+s-2) + (t-s)x - (x_0+x_0')(t-1) + e - \floor{\frac{(s-1)^2}{4}},    \nonumber
\end{align}
what we wanted to show.
\end{proof}
\vspace{-1em}
\end{proof}

We state the following immediate corollary of this claim, which we need in Section 4.

\begin{COR} \label{cor:sub}
If we have equality in Claim~\ref{lem:sub}, then the following statements hold:
\begin{itemize}
 \item any vertex in $B_1'\cup C'$ has exactly $t-1$ neighbors in $A\cup B$,
 \item any vertex in $B$ has exactly $s-1$ neighbors in $A'$,
 \item the vertices in $C_1$ have exactly $s-1$ neighbors outside $B_2'$, and
 \item $y(s-1)-yy'=\floor{(s-1)^2/4}$.
\end{itemize}
\end{COR}

Now we are ready to prove our general theorem, which is tight up to an additive constant. Let us emphasize, however, that since our methods do not give the exact result, we will not make any effort to optimize the constant error term.

\begin{THM} \label{thm:main}
If $G=(U,U',E)$ is a $K_{s,t}$-saturated bipartite graph with $n$ vertices on each side, then it contains at least $(s+t-2)n-(s+t-2)^2$ edges.
\end{THM}
\begin{proof}
Following Lemma~\ref{lem:main}, our plan is to find an appropriate core.

By Proposition \ref{prop:smalldeg}, we may assume that the minimum degree of our graph is at least $t-1$. Suppose for contradiction that $G$ contains fewer than $(s+t-2)n-(s+t-2)^2$ edges. Then there is a vertex $u_0\in U$ of degree at most $s+t-3$. Moreover, there is a non-adjacent vertex $u_0'\in U'-N(u_0)$ of degree at most $s+t-3$ as well, since otherwise the number of edges in $G$ would be at least $(n-(s+t-3))(s+t-2) > (s+t-2)n-(s+t-2)^2$, contradicting our assumption. Set $A_0= \{u_0\} \cup N(u_0')$ and $A_0'= \{u_0'\}\cup N(u_0)$, and define $\tilde{A_0}=A_0\cup A_0'$ to be the core.

Using the above notation, we see that $x_0 = |A_0| = 1+|N(u_0')| \le s+t-2$ and $x_0'=|A_0'|=1+|N(u_0)| \le s+t-2$. Since $u_0$ and $u_0'$ are not adjacent, we can add the edge $u_0u_0'$ to create a new $K_{s,t}$. Notice that all the vertices of this $K_{s,t}$ are adjacent to either $u_0$ or $u_0'$, hence they all lie in the core. Consequently, the core spans $e=e(A_0, A_0') \ge st-1$ edges. Now applying Lemma~\ref{lem:main} we get
\begin{align}
e(U,U') &\ge n(s+t-2) - (x_0+x_0')(t-1) + \min\{(t-s)x,(t-s)x'\} - \floor{\frac{(s-1)^2}{4}} + e \\
 &\ge n(s+t-2) - (x_0+x_0')(t-1) + \min\{(t-s)x_0,(t-s)x_0'\} - \floor{\frac{(s-1)^2}{4}} + st-1 \nonumber \\
 &\ge n(s+t-2) - (s+t-2)^2 + st-1 - \floor{\frac{(s-1)^2}{4}} \nonumber \\
 &\ge n(s+t-2) - (s+t-2)^2. \nonumber
\end{align}
This contradicts the assumption, thus proving the theorem.
\end{proof}

\section{Extremal graphs}  \label{sec:example}

As we mentioned in the introduction, Moshkovitz and Shapira \cite{MS12} constructed a $K_{s,t}$-saturated $n$-by-$n$ bipartite graph showing that
the bound of the Conjecture~\ref{conj:main}, if true, is tight. It appears that this example is not unique. In this section
we describe a general family of such graphs which contains the example by Moshkovitz and Shapira as a special case (when $l=1$).

\begin{example} As usual, we denote the two sides of the bipartite graph by $U$ and $U'$, where $|U|=|U'|=n$. Let us break each class into two major parts: $U=V\cup W$ and $U'=V'\cup W'$, where $|V|=|V'|= \floor{\frac{t+s-2}{2}}$ (assume $n$ is large enough). Suppose $W$ and $W'$ are further broken into some parts $W_1,\ldots, W_l$ and $W_1', \ldots, W_l'$ where $|W_i|=|W_i'|\ge t-s$ for all $i$. The construction of an extremal graph $G$ goes as follows.

First include in $G$ all the edges between $V$ and $V'$, making it a complete bipartite graph. Also, for every $i$, choose the edges between $W_i$ and $W_i'$ to span an arbitrary $(t-s)$-regular graph. It remains to describe the edges going between different type of classes.

We do not include any edge between $W_i$ and $W_j'$ for any $i\ne j$. Instead, choose arbitrary sets $S' \subs V'$ and $S_1, \ldots, S_l \subs V$ of size $s-1$, and take all edges going between $W_i$ and $S'$ as well as the edges between $S_i$ and $W_i'$, for all $i$. A straightforward computation shows that the number of edges in this $G$ is exactly the number in the conjecture. We claim that $G$ is $K_{s,t}$-saturated.

Let us see what happens when we add a missing edge $uu'$ to $G$. If $u'\in W'$, i.e. $u'\in W_i'$ for some $i$, then let $N$ be the set of its $t-s$ neighbors in $W_i$. Since $u\in U - N - S_i$, the set $S_i\cup \{u\}\cup N\cup S'\cup \{u'\}$ then forms a $K_{(t,s)}$. On the other hand, if $u'\in V'$, then $u\in W_i$ for some $i$. Let $N'$ be the set of the $t-s$ neighbors of $u$ in $W_i'$, then $u'\in U' - N' - S'$ and hence the set $S_i\cup \{u\}\cup S'\cup N'\cup \{u'\}$ forms a $K_{(s,t)}$. This proves the saturation property.
\end{example}

The asymmetric structure of the above example comes from the relaxation of the $l=1$ case, which corresponds to the construction of
Moshkovitz and Shapira. When all the vertices in $W'$ are connected to the same subset of $V$ of size $s-1$, adding an edge between
$W$ and $W'$ creates both a $K_{(s,t)}$ and a $K_{(t,s)}$. Our example exploits the freedom we had in choosing the edges between $W'$ and $V$.
In our case, when $l>1$, adding an edge between $W_i$ and $W_j'$ with $S_i\ne S_j$ creates only a $K_{(t,s)}$. The existence of such asymmetric
examples provides further difficulties in proving an exact result.

\section{The $K_{2,3}$ case}   \label{sec:case}

For $t=s$, Conjecture \ref{conj:main} trivially follows from the ordered result by Bollob\'as \cite{B67}. The other extreme is also easy to handle. When $s=1$, the $K_{1,t}$-saturated property merely means that the vertices of degree less than $t-1$ span a complete bipartite graph. Then it is a simple exercise to show (see \cite{MS12}) that the conjecture holds in this case as well.

Thus the first open case is $s=2$ and $t=3$, where the conjecture asserts that any $K_{2,3}$-saturated graph contains at least $3n-2$ edges. We note that there are many saturated graphs on $3n-2$ edges. In fact, there are many such examples which are even $K_{(2,3)}$-saturated: Just take a vertex $v'\in U'$ that is connected to everything in $U$, and make sure that every other vertex in $U'$ has degree 2.

In this section we prove the matching lower bound. A brief summary of the coming theorem can be phrased as follows. By finding an appropriate core, our techniques from Section \ref{sec:main} easily give a $3n-3$ lower bound. The rest of the proof is then a series of small structural observations, ultimately ruling out the possibility that a $K_{2,3}$-saturated graph with $3n-3$ edges exists.

\begin{THM}
If $G=(U,U'; E)$ is a $K_{2,3}$-saturated bipartite graph with $n\ge 4$ vertices in each part, then it has at least $3n-2$ edges.
\end{THM}
\begin{proof}
As a first step, we show in the spirit of Proposition \ref{prop:smalldeg} that it is enough to consider graphs of minimum degree 2.

\begin{LEMMA} \label{lem:lem1}
If $G$ contains fewer than $3n-2$ edges, then it has minimum degree 2. Moreover, it contains two non-adjacent vertices $u_0\in U$ and $u_0'\in U'$ of degree 2.
\end{LEMMA}
\begin{proof} The saturation property ensures that each vertex has at least one neighbor. Suppose there is a vertex $u$ of degree 1 -- wlog $u\in U$ -- and let $u'\in U'$ be its neighbor. Take any vertex $v'\in U'$ other than $u'$, then adding the edge $uv'$ cannot create a $K_{(2,3)}$, so it must create a $K_{(3,2)}$, with the 2-vertex class being $\{u',v'\}$. For any such $v'$, let $U_{v'}\subs U$ be the 3-class of this $K_{(3,2)}$, so $U_{v'}$ consists of $u$ and two neighbors of $v'$. We count the two edges between $v'$ and $U_{v'}$ for each $v'\in U'$, $v'\ne u'$ to get a total of $2n-2$ different edges.

Now let $X=\cup U_{v'}$, then every vertex in $X$ is connected to $u'$ because each of the above $K_{(3,2)}$'s contains $u'$. This gives $|X|$ new edges. On the other hand, we still have not encountered any edges touching $U-X$. But since we know that each vertex has at least one neighbor, we surely have at least $n-|X|$ new edges. This is already a total of $3n-2$ edges in $G$, contradicting our assumption.

Therefore the minimum degree is at least 2, but in fact it is exactly two, as otherwise we would have at least $3n$ edges in the graph. Let $u_0$ have degree 2 -- we may assume $u_0\in U$. If every non-adjacent vertex in $U'$ has at least 3 neighbors, then we have $2\cdot 2+ 3(n-2)=3n-2$ edges incident to $U'$, again a contradiction. Hence there is a $u_0'\in U'$ of degree 2 that is not adjacent to $u_0$, and we are done.
\end{proof}

Suppose $G$ is a counterexample to our theorem, and apply Lemma~\ref{lem:lem1} to get two non-adjacent vertices $u_0$ and $u_0'$ of degree 2. Denote the neighbors of $u_0$ by $u_1', u_2'\in U'$, the neighbors of $u_0'$ by $u_1, u_2\in U$, and let $A_0=\{u_0,u_1,u_2\}$ and $A_0'=\{u_0',u_1',u_2'\}$ be the core of the structure we described at the beginning of Section~\ref{sec:main}. Using this core we will also construct the sets $A$, $B=B_1\cup B_2$, $C=C_1\cup C_2$ and $A'$, $B'=B_1'\cup B_2'$, $C'=C_1'\cup C_2'$ as defined by the structure.

Assume that $|C_2|\le |C_2'|$ and apply Claim~\ref{lem:sub} with $s=2$ and $t=3$ to the structure of core $\tilde{A_0}=A_0\cup A_0'$. These choices for $s$ and $t$ significantly simplify the bound we get from this claim:
\[ e(U,U')\ge 3n - 6\cdot 2 - 0 + e + x = 3n-12+e+x. \]
Using that the addition of the edge $u_0u_0'$ creates a $K_{2,3}$ inside the core (as all neighbors of $u_0$ and $u_0'$ are in $\tilde{A_0}$), it is easy to check that $e=e(A_0,A_0')\ge 6$. We also know that $x\ge x_0=3$, so $e(U,U')\ge 3n-3$. Then these inequalities together with Corollary~\ref{cor:sub} imply that if $e(U,U') = 3n-3$ then $G$ satisfies the following five properties:

\begin{enumerate}
 \item $e=e(A_0,A_0')=6$ and $x=|A|=3$,
 \item any vertex in $B_1'\cup C'$ has exactly $2$ neighbors in $A\cup B$,
 \item any vertex in $B$ has exactly $1$ neighbor in $A'$,
 \item the vertices in $C_1$ have exactly $1$ neighbor outside $B_2'$, and
 \item for $y=|C_2|$ and $y'=|C_2'|$ (with $0\le y\le y'$) we have $y(s-1)-yy'=y(1-y')=0$, so either $y=0$ or $y=y'=1$.
\end{enumerate}

The following lemma supplements the fifth property and shows that $C_2$ must be empty and $C_2'$ must be non-empty, by taking care of the case $y=y'=0$ and $y=y'=1$.

\begin{LEMMA} \label{lem:sym}
If $|C_2|=|C_2'|$ then $G$ spans at least $3n-2$ edges.
\end{LEMMA}
\begin{proof} As $G$ is a counterexample, by Lemma~\ref{lem:lem1} it has minimum degree 2. Since $y=y'$, we may apply Claim~\ref{lem:sub} and Corollary~\ref{cor:sub} to $G$'s ``mirror'', with $U$ and $U'$ switched, and observe that the five properties hold for this mirror graph as well. Then the first property gives $x=3$, $x'=3$ and $e=6$. So $A=\{u_0,u_1,u_2\}$ and $A'=\{u_0',u_1',u_2'\}$ (i.e. any vertex not in the core has at most one neighbor in it), and the core spans 6 edges. By symmetry we can assume that adding the edge $u_0u_0'$ creates a $K_{(2,3)}$ on the set $\{u_0,u_1,u_0',u_1',u_2'\}$, so the missing edges are $u_0u_0'$, $u_2u_1'$ and $u_2u_2'$. We also assumed that there is no vertex of degree 1, so $u_2$ must have some neighbor $v'$ in $B'$. Note that $v'$ has exactly one neighbor in $A$, in particular it is not connected to $u_1$.

Now let us see what happens when we add the edge $u_0v'$. We cannot create a $K_{(2,3)}$, because that would use both $u_1'$ and $u_2'$, but their only common neighbor other than $u_0$ is $u_1$ (recall that no vertex outside $A'$ can have $2$ neighbors in $A$), which is not connected to $v'$. So it must be a $K_{(3,2)}$, and it is not using $u_2$, as $u_2$ has no common neighbor with $u_0$. But then the $K_{(3,2)}$ contains two neighbors of $v'$ that are not in $A$, but are connected to a vertex in $A'$.
Then, by definition, these neighbors are in $B$. So $v'\in B_1'$ has at least two neighbors in $B$ and one in $A$, and this contradicts the second property.
\end{proof}

From now on we assume that $C_2$ is empty and $C_2'$ is non-empty. Then the fourth property also implies that the vertices in $C=C_1$ have exactly one neighbor outside $B_2'$. Moreover, the third property tells us that each vertex in $B$ has exactly one neighbor in $A'$.

\begin{LEMMA} \label{lem:lem3}
All vertices in $B$ are connected to the same vertex in $A'$.
\end{LEMMA}
\begin{proof}
Break $B$ into parts based on the neighbor in $A'$ by putting the vertices in $B$ connected to $w'\in A'$ into the set $B_{w'}$. We claim that vertices in different parts do not share common neighbors, or in other words, any vertex $v'\in B'\cup C'$ has all its neighbors in $B$ contained in the same part $B_{w'}$.

Indeed, any vertex in $B'$ has at most one neighbor in $B$: this is true by definition for the vertices in $B_2'$, and follows from the second property for $B_1'$ (every vertex in $B'_1$ has a neighbor in $A$). Now look at the vertices in $C'$. An easy observation in Lemma~\ref{lem:main} shows that adding the edge $u_0v'$ for $v'\in C'$ cannot create a $K_{(2,3)}$. So it creates a $K_{(3,2)}$, and this $K_{(3,2)}$ must contain the two neighbors of $v'$ in $B$ and a neighbor $w_0'$ of $u_0$ in $A'$. Hence both neighbors of $v'$ are in $B_{w_0'}$, establishing the claim.

As we noted above, $C'$ is not empty, so take a vertex $v_1'\in C'$ and assume that the neighbors of $v_1'$ in $B$ are in $B_{w_1'}$. We will show that $B=B_{w_1'}$. Suppose not, i.e. there is a $v_1\in B_{w_2'}$ with $w_1'\ne w_2'$. Then the edge $v_1v_1'$ is missing; let us see what happens when we add that edge. We create a $K_{(2,3)}$ or a $K_{(3,2)}$, so in any case there are vertices $v_2\in U$ and $v_2'\in U'$ such that $v_1v_2'$, $v_2v_2'$ and $v_2v_1'$ are all edges of $G$. Here $v_2$ cannot be in $A$, as it is connected to $v_1'\in C'$. It is not in $B$ either, since then one of $v_1'$ and $v_2'$ would have neighbors in both $B_{w_1'}$ and $B_{w_2'}$. So $v_2\in C=C_1$ (since $C_2$ is empty). Now the fourth property says that $v_2$ has exactly 1 neighbor outside $B_2'$. Since $v_1' \in C'$, $v_2'$ must be in $B_2'$. But the vertices in $B_2'$ have no neighbors in $B$ so $v_1v_2'$ cannot be an edge, giving a contradiction.
\end{proof}

Note that this lemma implies that one of the two neighbors of $u_0$ -- say $u_1'$ -- is not connected to any vertex in $B$, and therefore it is adjacent to at least two vertices in $A=A_0$. We also recall that the core only spans six edges.
It is time to analyze what happens in the core when we add the edge $u_0u_0'$. It might create a $K_{(3,2)}$ or a $K_{(2,3)}$, but the obtained graph is inside the core in both cases.

\textbf{Case 1:} $u_0u_0'$ creates a $K_{(3,2)}$.

\vspace{0.1cm}
\noindent
If this $K_{(3,2)}$ used $u_2'$, then the core of $G$ would contain more than 6 edges: 5 from the $K_{(3,2)}$ and 2 other edges incident to $u_1'$, which is impossible. So $u_1'$ is connected to both $u_1$ and $u_2$, while $u_2'$ is not connected to any of them. Note, however, that $u_2'$ is connected to all vertices in $B$.

Let $v$ be any vertex in $U-A$. When we add the edge $vu_0'$ to $G$, we create a $K_{(2,3)}$ or a $K_{(3,2)}$, so there is a vertex $v'$ connected to both $v$ and $u_1$ or $u_2$. Then $v'$ is not in $A'$, since $v$ is only connected to $u_2'$ in $A'$, but both $u_1u_2'$ and $u_2u_2'$ are missing. Thus $v'\in B'$ (it has a neighbor in $A$, so it is not in $C'$). When we add the edge $u_0v'$, we cannot create a $K_{(2,3)}$, because that would use both $u_1'$ and $u_2'$, which only share $u_0$ as their common neighbor. So it creates a $K_{(3,2)}$ using one of $u_1'$ and $u_2'$. It cannot be $u_1'$, because then the 3-class of the $K_{(3,2)}$ is exactly $A$, making $v'$ have 2 neighbors in $A$. Thus, by definition, $v'\in A'$ which contradicts $v' \in B'$. But it cannot be $u_2'$ either, because then $v'$ would have two neighbors in $B$, which together with a neighbor in $A$ that $v'$ must have, contradicts the second property. So this case is impossible.

\textbf{Case 2:} $u_0u_0'$ creates a $K_{(2,3)}$.

\vspace{0.1cm}
\noindent
Then one of $u_1$ and $u_2$ -- say $u_1$ -- is connected to both $u_1'$ and $u_2'$, and the other is connected to neither. But then $u_1'$ has exactly two neighbors, $u_0$ and $u_1$, and the set $\bar{A}=\{u_0,u_1,u_1',u_2'\}$ spans four edges. This means that we can apply Lemma~\ref{lem:main} taking $\bar{A}$ as the core, and $u_0$ and $u_1'$ being its ``distinguished'' vertices with their neighborhoods also sitting in the core. One can check that all conditions are satisfied, and with our new values of $x\ge x_0=2$, $x'\ge x_0'=2$ and $e=4$, we get that $G$ has at least $3n-8-0+4+2=3n-2$ edges. This contradiction finishes the proof of the theorem.
\end{proof}

\section{Concluding remarks}
\label{sec:conclusion}

Although we could slightly improve the error term in Theorem~\ref{thm:main}, it seems that more ideas are needed to
prove the full conjecture. We also note that our methods can be used to provide an asymptotically tight estimate on
the minimum number of edges in a $K_{s,t}$-saturated unbalanced bipartite graph (i.e., with parts of size $m$ and $n$).
Determining the precise value in this unbalanced case might be even more challenging, although we believe that a straightforward modification of the extremal construction from the balanced case is tight here as well.

Bipartite saturation results were generalized to the hypergraph setting in \cite{A85,MS12}, where $G$ and $H$ are assumed to be $k$-partite $k$-uniform hypergraphs, and $G$ is $H$-saturated if any new hyperedge meeting one vertex from each color class creates a new copy of $H$. It would be interesting to extend our results to get an asymptotically tight bound for the unordered $k$-partite hypergraph saturation problem.

\vspace{0.3cm}
\noindent
{\bf Acknowledgment.} We would like to thank A. Shapira for bringing this problem to our attention.

\end{document}